\documentclass[12pt]{amsart}
\usepackage{amssymb,amsbsy,amsthm,amscd}
\usepackage[mathcal]{eucal}

\usepackage[all]{xy}

\theoremstyle{plain}
\newtheorem{thm}{Theorem}[section]

\newtheorem{lem}[thm]{Lemma}

\newtheorem{cor}[thm]{Corollary}
\newtheorem{prop}[thm]{Proposition}

\numberwithin{equation}{section}
\newcommand{\I}{\ensuremath{\mathbf I}}
\newcommand{\N}{\ensuremath{\mathbb N}}
\newcommand{\E}{\ensuremath{\mathbf E}}
\newcommand{\Z}{\ensuremath{\mathbb Z}}

\DeclareMathOperator{\z}{Z}
\DeclareMathOperator{\ab}{ab}
\DeclareMathOperator{\lcm}{lcm}

\DeclareMathOperator{\Inn}{Inn}
\DeclareMathOperator{\Out}{Out}
\DeclareMathOperator{\GL}{GL}
\DeclareMathOperator{\SL}{SL}

\DeclareMathOperator{\Aut}{Aut}

\DeclareMathOperator{\ord}{ord}

\DeclareMathOperator{\id}{id}

\begin{document}
\title{On the Index of\\Congruence Subgroups of $\Aut(F_n)$}
\thanks{The first author was supported by the Thomas Holloway Scholarship. The second author was supported by the DFG (German Research Foundation). The authors would like to thank F.~Grunewald for proposing this very interesting topic and B.~Klopsch for many useful discussions.}
\author{Daniel Appel}
\address{Mathematisches Institut der Heinrich- Heine- Universit\"at\\ 40225 D\"usseldorf\\ Germany.}
\email[D.~Appel]{Daniel.Appel@uni-duesseldorf.de}
\author{Evija Ribnere} 
\email[E.~Ribnere]{Evija.Ribnere@uni-duesseldorf.de}

\maketitle
\pagestyle{plain}
\begin{abstract}
For an epimorphism $\pi : F_n \rightarrow G$ of the free group $F_n$ onto a finite group $G$ we call $\Gamma(G,\pi)=\{ \varphi \in \Aut(F_n) \mid \pi\varphi~=~\pi \}$ the standard congruence subgroup of $\Aut(F_n)$ associated to $G$ and $\pi$.
In the case $n=2$ we present formulas for the index of $\Gamma(G,\pi)$ where $G$ is abelian or
dihedral. Moreover, we show that congruence subgroups associated to dihedral groups provide a family of subgroups of arbitrary large index in $\Aut(F_2)$ generated by a fixed number of elements. This implies that finite index subgroups of $\Aut(F_2)$ cannot be written as free products.
\end{abstract}

\section{Introduction}
\subsection{Main Results}
Let $F_n$ be the free group on $n$ generators and $\Aut(F_n)$ its
group of automorphisms. Moreover, let $\pi: F_n \rightarrow G$ be an
epimorphism of $F_n$ onto a finite group $G$ and let $R$ be its
kernel. As in \cite{GL} we define
$$\Gamma(R) := \{ \varphi \in \Aut(F_n) \mid \varphi(R) = R \}.$$
Every $\varphi \in \Gamma(R)$ induces an automorphism of $F_n / R \cong G$. We call
\begin{eqnarray*}
\Gamma(G,\pi) &:=& \{ \varphi \in \Gamma(R) \mid \varphi \mbox{ induces the identity on } F_n / R \}\\
& = &\{ \varphi \in \Aut(F_n) \mid  \pi \varphi =
\pi\}
\end{eqnarray*}
the \emph{standard congruence subgroup of $\Aut(F_n)$ associated to
$G$ and $\pi$}. These subgroups are of finite index in $\Aut(F_n)$ (see Subsection~\ref{secAct}). A subgroup of $\Aut(F_n)$ containing some $\Gamma(G,\pi)$ is
called a \emph{congruence subgroup of $\Aut(F_n)$}. We denote by $\Aut^+(F_n)$ the special automorphism group of $F_n$ (see Subsection~\ref{comments} for details) and  write $\Gamma^+(G,\pi) := \Gamma(G,\pi) \cap \Aut^+(F_n)$. The term \emph{congruence subgroup of} $\Aut^+(F_n)$ is defined in the obvious way.

In \cite{GL} Grunewald and Lubotzky use the groups $\Gamma(G,\pi)$ to construct linear representations of the automorphism group $\Aut(F_n)$. In their concluding Section 9.4 they present, for some explicit $G$, the indices of the groups $\Gamma^+(G,\pi)$ in $\Aut^+(F_n)$, which are determined by MAGMA computations. However, their only general result in this context is
$$[\Aut^+(F_n) : \Gamma^+(\Z / 2 \Z,\pi)] = 2^n - 1.$$

In this paper we provide a first step towards a systematic study of the groups $\Gamma^+(G,\pi)$ and especially their indices in $\Aut^+(F_n)$. We focus on the case $n = 2$ and $G$ abelian or dihedral. Our main results are

\begin{thm}\label{thm-abel}
Let $G$ be a finite abelian group and $\pi: F_2 \rightarrow G$ an arbitrary epimorphism. Writing $G  \cong \Z/m\Z \times \Z/n\Z$ with $n \mid m$ one has
$$[\Aut^+(F_2) : \Gamma^+(G,\pi)] = n m^2 \prod_{p \mid m} (1-\frac{1}{p^2})$$
 where the product runs over all primes $p$ dividing $m$.
\end{thm}

\begin{thm}\label{thm-dieder}
Let $\pi: F_2 \rightarrow D_n$ be an arbitrary epimorphism of $F_2$ onto the dihedral group $D_n$.
Then
$$[\Aut^+(F_2) : \Gamma^+(D_n,\pi)] =6n.$$
Moreover, $\Gamma^+(D_n,\pi)$ is generated by four elements.
\end{thm}
The Reidemeister method (see for example \cite{MKS}) implies
\begin{cor}
Any group commensurable with $\Aut(F_2)$, contains subgroups of arbitrary large index, generated by a fixed number of elements.

In particular, finite index subgroups of $\Aut(F_2)$ cannot be written as free products.
\end{cor}
The fact that $\Aut(F_2)$ finite index subgroups of $\Aut(F_2)$ cannot be written as free products follows from the Kurosh Subgroup Theorem \cite{DR}. Observe that the special linear group $\SL_2(\Z)$ behaves in this respect very differently from the special automorphism group $\Aut^+(F_2)$. For a bounded number of generators we cannot obtain subgroups of arbitrary large index in $\SL_2(\Z)$. Moreover, $\SL_2(\Z)$ contains the finite index subgroup $\langle \left(\begin{smallmatrix}1 & 2 \\ 0 & 1 \end{smallmatrix}\right), \left(\begin{smallmatrix}1 & 0 \\ 2 & 1 \end{smallmatrix}\right) \rangle$ which is free of rank $2$.

\subsection{Comparison with $\SL_n(\Z)$}
Let us describe the analogy between congruence subgroups of $\Aut^+(F_n)$ and congruence subgroups of $\SL_n(\Z)$. A group of the form
$$\Gamma(m)= \{ \phi \in \SL_n(\Z)\mid \phi \equiv_m \I_n \},
$$ where $m \in \N$ and $\I_n$ denotes the identity matrix, is called a \emph{principal congruence subgroup} of $\SL_n(\Z)$. A subgroup of $\SL_n(\Z)$ containing some $\Gamma(m)$ is called a \emph{congruence subgroup}. Note that $\SL_n(\Z)$ is in fact a subgroup of index $2$ of the automorphism group $\GL_n(\Z)$ of the free abelian group $\Z^n$. Consider the natural epimorphism $\Z^n \rightarrow (\Z/m\Z)^n$. Its kernel $(m\Z)^n$ is invariant under every automorphism $\phi \in \SL_n(\Z)$, so that every $\phi \in \SL_n(\Z)$ induces an automorphism of $(\Z / m\Z)^n$. One easily sees that
$$\Gamma(m)=\{ \phi \in \SL_n(\Z) \mid \phi \mbox{ induces the identity on } (\Z/m\Z)^n  \}.$$

\subsection{Detailed Discussion of Results and Strategies of the Proofs}\label{comments}
The automorphism group $\Aut(F_n)$ has  a well-known surjective representation
$$\rho : \Aut(F_n) \rightarrow \Aut(F_n / F'_n) \cong \GL_n(\Z),$$
where $F'_n$ denotes the commutator subgroup of $F_n$.  Its kernel is denoted by $T_n$ and called the
\emph{Torelli group}. As one classically considers $\SL_n(\Z)$ instead of $\GL_n(\Z)$, we shall focus on the \emph{special automorphism group}
$\Aut^+(F_n) := \rho^{-1}(\SL_n(\Z))$, which is a subgroup of index
2 in $\Aut(F_n)$ (see for example \cite{MKS}). We also
set
$$\Gamma^+(G,\pi) := \Gamma(G,\pi) \cap \Aut^+(F_n).$$
This is a subgroup of index at most 2 in $\Gamma(G,\pi)$. Note that $T_n \leq \Aut^+(F_n)$.

Using the representation $\rho$ we can write the index of $\Gamma^+(G,\pi)$ in
$\Aut^+(F_n)$
as a product of two other indices which are easier to compute.
 See Subsection~\ref{index-allgemein} for the proof.
\begin{prop}\label{intr-prop-index}Let $\pi : F_n \rightarrow G$ be an epimorphism of $F_n$ onto a finite group $G$. Then
$$[\Aut^+(F_n) : \Gamma^+(G,\pi)] = [\SL_n(\Z) : \rho(\Gamma^+(G,\pi))] \cdot [T_n : T_n \cap \Gamma^+(G,\pi)].$$
\end{prop}

For the reminder we consider the case $n=2$. A
classical result of Nielsen (see for example  \cite{MKS}) says that
in this case the Torelli group is exactly the group of inner automorphisms, i.e.,\ $T_2 = \Inn(F_2).$ 
This enables us to prove in Section~\ref{sec-preliminaries} that the quotient group $T_2 / T_2 \cap \Gamma^+(G,\pi)$ is isomorphic to $G / \z(G)$, where $\z(G)$ denotes
the center of $G$.
Hence, for $n=2$ we can derive the following from Proposition
\ref{intr-prop-index}.
\begin{cor}\label{intr-cor-index-n2}Let $\pi : F_2 \rightarrow G$ be an epimorphism of $F_2$ onto a finite group $G$. Then
$$[\Aut^+(F_2) : \Gamma^+(G,\pi)] = [\SL_2(\Z) : \rho(\Gamma^+(G,\pi))] \cdot [G : \z(G)].$$
\end{cor}

In Section~\ref{sec-abelian} we use the above result to determine the index of $\Gamma^+(G,\pi)$ in $\Aut^+(F_2)$ for abelian groups $G$ and thus prove Theorem~\ref{thm-abel}. Note that in this case the index depends only on $G$, but not on $\pi$. To see
this, we  prove that for any two epimorphisms $\pi_1,\ \pi_2: F_2
\rightarrow G$, the congruence subgroups $\Gamma^+(G,\pi_1)$ and
$\Gamma^+(G,\pi_2)$ are conjugate in $\Aut^+(F_2)$. Since $G$ is abelian, we have $[G
: \z(G)] = 1$. Hence, by Corollary~\ref{intr-cor-index-n2} we only
need to determine the index $[\SL_2(\Z) : \rho(\Gamma^+(G,\pi))]$
for some convenient choice of $\pi$. We can choose $\pi$ such that
$\rho(\Gamma^+(G,\pi))$ is the classical congruence subgroup
$$\Gamma(m,n)= \{ \left(
\begin{smallmatrix} a & b \\ c & d\end{smallmatrix} \right) \in
\SL_2(\Z)\mid a \equiv_m 1 , b \equiv_m 0  \text{ and } c \equiv_n
0 , d \equiv_n 1 \}$$
with $n \mid m$, whose index is described in Lemma~\ref{ind_gamma}.
From our discussions we can easily derive the index of the
classical congruence subgroup $\Gamma(m,n)$ for arbitrary $m$ and
$n$ (see Section~\ref{sec-abelian}).

Finally, in Section~\ref{sec-dihedral} we consider the case that $G$ is a
dihedral group and prove Theorem~\ref{thm-dieder}. 

\subsection{Conjectures, Remarks and Related Problems}
\begin{enumerate}
\item Let $G$ be the non-abelian semidirect product of two cyclic groups, $G=\Z/p\Z \ltimes \Z/q\Z $, where $p$ and $q$ are primes with $q \equiv_p 1$.
We conjecture that then the index of $\Gamma^+(G,\pi)$ in $\Aut^+(F_2)$ is
$$|G| \cdot [\SL_2(\Z):\Gamma(p,1)]=pq p^2 (1-\frac{1}{p^2})=qp(p^2-1).$$
For $p=2$ this coincides with the formula in Theorem~\ref{thm-dieder}.

\item The congruence subgroup problem: is every finite-index subgroup of $\Aut^+(F_n)$ a congruence subgroup? For $\SL_2(\Z)$ this problem is already solved. The answer is yes for $n \geq 3$ (see \cite{BLS}, \cite{M}) and no for $n = 2$ (see \cite{FK}). However, it is still not clear what the answer for $\Aut^+(F_n)$ should be.

Let us state some partial results for the case $n = 2$. So far we can say that there are finite-index subgroups of $\Aut^+(F_2)$ that do not contain any $\Gamma^+(G,\pi)$, with $G$  abelian or dihedral (see Section~\ref{sec-congr}). However, from Asada's results in \cite{MA} it follows that every finite index subgroup of $\Aut^+(F_2)$ containing $\Inn(F_2)$ is a congruence subgroup. To be more precise, Asada shows that every finite-index subgroup of $\Aut^+(F_2)/ \Inn(F_2) =: \Out^+(F_2)$ contains some group of the form 
$$\ker(\Out^+(F_2) \rightarrow \Out(F_2/K))$$
where $K \leq F_2$ is a characteristic subgroup of $F_2$.\\

\item For which $G$ and $\pi$ is the image $\rho(\Gamma^+(G,\pi))$ a congruence subgroup of $\SL_2(\Z)$? For abelian or dihedral groups $G$ it always is, but in general this is not true. A counterexample is given by $G=A_5$, the alternating group of degree $5$. Moreover, $A_5$ is the smallest group with this property.\\

\item As a generalisation of the abelian case, one might expect that $\rho(\Gamma^+(G,\pi))$ is always a congruence subgroup, if $G$ is solvable. This turns out to be false. We found a solvable group $G$ of order $128$ for which $\rho(\Gamma^+(G,\pi))$ is not a congruence subgroup of $\SL_2(\Z)$ (see Section~\ref{sec-congr} for details). Computational results indicate that $\rho(\Gamma^+(G,\pi))$ is always a congruence subgroup, if $G$ is metabelian.\\

\item The group $\Aut^+(F_2)$ acts in a natural way on the set $\mathbf{R}_2(G):=\{\ker(\pi) \mid \pi: F_2 \rightarrow G \text{ epimorphism} \}$ (see Subsection \ref{secAct}). This leads to a classical question that was first asked by W. Gasch\"utz and B. H. Neumann (1950s): for which finite groups $G$ is this action  transitive?
The answer is of importance to us, because, up to conjugation,
$\Gamma^+(G,\pi)$ depends only on the $\Aut^+(F_2)$-orbit of
$\ker(\pi)$ in $\mathbf{R}_2(G)$. If $G$ is abelian or dihedral,
the action is transitive, but for $G = A_5$ it is not. Indeed,
different choices for $\pi : F_2 \rightarrow A_5$ lead to congruence
subgroups of different indices. See also \cite[Section 9.1]{GL} for
more comments on this problem.
\end{enumerate}
\section{Preliminaries}\label{sec-preliminaries}

\subsection{Congruence Subgroups of $\SL_2(\Z)$}
Let $\pi : F_2 \rightarrow G$ be an epimorphism of the free group $F_2$ onto a finite group $G$. As the image $\rho(\Gamma^+(G,\pi))$ is a finite-index subgroup of $\SL_2(\Z)$, we recall the notation for congruence subgroups of $\SL_2(\Z)$.
For $m,n \in \N$ let
$$\Gamma(m,n) := \{\left( \begin{smallmatrix} a & b \\ c & d\end{smallmatrix} \right) \in \SL_2(\Z) \mid a \equiv_m 1, b \equiv_m 0, c \equiv_n 0, d \equiv_n 1 \}.$$
Then the principal congruence subgroup $\Gamma(m)$ is exactly $\Gamma(m,m)$. One also writes $\Gamma^1(m) := \Gamma(m,1)$ and $\Gamma_1(n) := \Gamma(1,n)$.

In our proofs we need the indices of these subgroups in $\SL_2(\Z)$. They are known for $\Gamma^{1}(m)$, $\Gamma_{1}(m)$ and $\Gamma(m)$ (see for example \cite[1.2]{DSh}):
\begin{align*}
[\SL_2(\Z) : \Gamma^1(m)] = & \ [\SL_2(\Z) : \Gamma_1(m)] =  m^2 \prod_{\substack{p \mid m\\ p \ \mathrm{prime}}} (1-\frac{1}{p^2}),\\
[\SL_2(\Z) : \Gamma(m)] = & \ m^3 \prod_{\substack{p \mid m\\ p \ \mathrm{prime}}} (1-\frac{1}{p^2}).
\end{align*}
However, the literature does not seem to include a formula for the
index of $\Gamma(m,n)$ for general $m,n \in \N$. As we shall see, we only need it for the case
that $n \mid m$ and we provide it in the next lemma. A formula for the index of $
\Gamma(m,n)$ for arbitrary $m$ and $n$ is given at the end of Section
\ref{sec-abelian}.
\begin{lem}\label{ind_gamma}
Let $m,n \in \N$ such that $n \mid m$. Then
$$[\SL_2(\Z) : \Gamma(m,n)] = n m^2 \prod_{p \mid m} (1-\frac{1}{p^2}),$$ where the product runs over all primes $p$ dividing $m$.
\end{lem}
\begin{proof}
If $A \in \Gamma^1(m)$, then $A \equiv \left( \begin{smallmatrix} 1 & 0 \\ * & 1\end{smallmatrix} \right)$ modulo $m$. Since $n \mid m$ this implies $A \equiv \left( \begin{smallmatrix} 1 & 0 \\ * & 1\end{smallmatrix} \right)$ modulo $n$. It is now easily seen that the matrices $ \left( \begin{smallmatrix}1 & 0 \\ k & 1 \end{smallmatrix} \right)$, $0 \leq k \leq n-1$, provide a coset representative system for $\Gamma(m,n)$ in $\Gamma^1(m)$ so that $[\Gamma^1(m) : \Gamma(m,n)] = n$. The lemma follows.
\end{proof}

\subsection{A Presentation of $\Aut^+(F_2)$}\label{subsec-repres}
We use the fact that the group $\Aut^+(F_2)$ is an extension of $T_2 = \Inn(F_2)$ by $\SL_2(\Z)$, i.e.\ the sequence
$$1  \longrightarrow  T_2  \longrightarrow  \Aut^+(F_2)  \stackrel{\rho}{\longrightarrow}  \SL_2(\Z) \longrightarrow  1$$
is exact. For an element $w \in F_2$ let $\alpha_w \in T_2$ be the inner automorphism of $F_2$ given by $\alpha_w(z) = wzw^{-1}$ for all $z \in F_2.$ The group $T_2$ is free on $\alpha_x$ and $\alpha_y$. Further, the special linear group $\SL_2(\Z)$ has a presentation
$$\SL_2(\Z)=\langle e_1, e_2 \: \vert \: e_2 e^{-1}_1 e_2 e_1 e^{-1}_2 e_1,
(e_2 e^{-1}_1 e_2)^4 \rangle, $$
where $e_1$ and $e_2$ represent $\left( \begin{smallmatrix} 1 & 0 \\ 1 & 1 \end{smallmatrix}\right)$ and
$\left( \begin{smallmatrix} 1 & 1 \\ 0 & 1 \end{smallmatrix}\right)$, respectively. Observe that preimages of $e_1$ and $e_2$ under $\rho$ are given by
$$  u = \begin{cases} x \mapsto xy \\ y \mapsto y \end{cases} \quad \mbox{and} \quad v = \begin{cases} x \mapsto x \\ y \mapsto xy, \end{cases}$$
respectively.
By a result of Hall \cite[Ch.~ 13, Th.~1]{DJ} we can compute the following presentation.
$$
\begin{array}{llll}
\Aut^+(F_2)= \langle \alpha_x, \alpha_y, u, v \: \vert \: 
 &  u \alpha_x u^{-1} = \alpha_x \alpha_y,  & u \alpha_y u^{-1} = \alpha_y, &  \\
 &  v \alpha_x v^{-1} = \alpha_x, & v \alpha_y v^{-1} = \alpha_x \alpha_y, & \\
 &  \multicolumn{2}{l}{v u^{-1} v u v^{-1} u = 1 , }  &  \\
 &  \multicolumn{2}{l}{(v u^{-1} v)^4= \alpha_x \alpha_y^{-1} \alpha_x^{-1} \alpha_y }& \rangle.
 \end{array}
 $$

\subsection{Dependence on the Epimorphism} \label{secAct}
For a finite group $G$ and $n \in \N$ we set
$$\E_n(G) := \{ \pi : F_n \rightarrow G \mid \pi \mbox{ is an epimorphism} \}.$$
Observe that $\E_n(G)$ is a finite set. The group $\Aut(G) \times \Aut(F_n)$ acts on this set by
$$(\phi, \varphi) \cdot \pi := \phi \pi \varphi^{-1} \quad \mbox{for} \quad \, \phi \in \Aut(G), \varphi \in \Aut(F_n), \pi \in \E_n(G).$$
Then $\Gamma(G,\pi)$ is exactly the stabiliser of $\pi$ under the action of $\Aut(F_n)$. Hence the orbit-stabiliser theorem yields
$$[\Aut(F_n) : \Gamma(G,\pi)] = |\Aut(F_n) \cdot \pi|.$$
In particular, $\Gamma(G,\pi)$ has finite index in $\Aut(F_n)$.
Moreover, up to conjugation, $\Gamma(G,\pi)$ only depends on the orbit of $\pi$ under this action.
Since, as it is easily seen, $\Gamma(G,\pi)$ is invariant under the action of $\Aut(G)$, we consider the set $\Aut(G) \backslash \mathbf{E}_n(G)$, which can be naturally identified with
$$\mathbf{R}_n(G) := \{ \ker(\pi) \mid \pi \in \E_n(G) \}.$$ 
The induced action of $\Aut(F_n)$ on this set is given by
$$\varphi \cdot R := \varphi(R) \quad \mbox{for} \quad \varphi \in \Aut(F_n), R \in \mathbf{R}_n(G).$$
Indeed, if $R = \ker(\pi)$, then $\varphi(R) = \ker(\pi \varphi^{-1})$. It follows that, up to conjugation, $\Gamma(G,\pi)$ depends only on the orbit of $\ker(\pi)$ in $\mathbf{R}_n(G)$.
	
We remark that the analogous results to the ones in this subsection also hold for $\Gamma^+(G,\pi)$ and $\Aut^+(F_n)$ replacing $\Gamma(G,\pi)$ and $\Aut(F_n)$, respectively.

\subsection{A Reduction to the Abelian Case}\label{reduc}
As before, let $G$ be a finite group and $\pi:F_n \rightarrow G$ an epimorphism. We naturally obtain an epimorphism $\bar{\pi}: F_n \stackrel{\pi}{\rightarrow} G \rightarrow G/G' = G^{\ab}$. If we have $\pi\varphi = \pi$ for some $\varphi \in \Aut(F_n)$, then clearly $\bar{\pi} \varphi = \bar{\pi}$ so that
$$\Gamma(G,\pi) \leq \Gamma(G^{\ab},\bar\pi).$$

\subsection{Proof of Proposition \ref{intr-prop-index}}\label{index-allgemein}
\begin{lem}
Let $A,B,C$ be groups with subgroups $A_0, B_0, C_0$, respectively. Assume we have a commutative diagram with exact rows
$$ \xymatrix{1 \ar[r] &  A  \ar[r]^\alpha   & B \ar[r]^\beta                    & C \ar[r] & 1\\
             1 \ar[r] & A_0 \ar@{^{(}->}[u] \ar[r]^{\alpha_0} & B_0 \ar@{^{(}->}[u] \ar[r]^{\beta_0} & C_0 \ar[r] \ar@{^{(}->}[u] & 1 }$$
where the homomorphisms from the second row to the first one are the inclusion maps and $\alpha_0$, $\beta_0$ are the restrictions $\alpha|_{A_0}$, $\beta|_{B_0}$, respectively. Assume further that $B_0$ has finite index in $B$. Then the indices $[A:A_0]$ and $[C:C_0]$ are also finite and we have
$$[B:B_0] = [A:A_0] \cdot [C:C_0].$$
\end{lem}

\begin{proof}
This result can be verified by diagram chasing. A complete proof will be contained in the Ph.D. thesis of the first author.
\end{proof}

Let us consider the following commutative diagram where $\rho$ is the representation introduced in Subsection~\ref{comments}.
$$ \xymatrix{1 \ar[r] &  T_n  \ar[r]    & \Aut^+(F_n) \ar[r]^\rho                   & \SL_n(\Z) \ar[r] & 1\\
             1 \ar[r] & T_n \cap \Gamma^+(G,\pi) \ar@{^{(}->}[u] \ar[r] & \Gamma^+(G,\pi) \ar@{^{(}->}[u] \ar[r]^\rho & \rho(\Gamma^+(G,\pi)) \ar[r] \ar@{^{(}->}[u] & 1 }$$
The rows of this diagram are exact and the homomorphisms from the second row to the first one are simply the inclusions. Applying the above lemma to this diagram, we obtain Propostion~\ref{intr-prop-index}.

The following result for the special case $n = 2$ leads to Corollary~\ref{intr-cor-index-n2}. Recall that $\z(G)$ denotes the center of $G$.
\begin{lem}\label{LemCen}
There is an exact sequence
$$ 1 \longrightarrow T_2 \cap \Gamma^+(G,\pi) \longrightarrow T_2 \longrightarrow \Inn(G) \longrightarrow 1.$$
In particular $[T_2 : T_2 \cap \Gamma^+(G,\pi)] = |\Inn(G)| = [G : \z(G)]$.
\end{lem}

\begin{proof}
For $g \in G$ we define $c_g \in \Inn(G)$ by $c_g(h) = ghg^{-1}$ for all $h \in G$. Let $\Phi : T_2 \rightarrow \Inn(G)$ be the homomorphism given by $\Phi(\alpha_z) := c_{\pi(z)}$ for all $z \in F_2$. Since $\pi : F_2 \rightarrow G$ is onto, it follows that $\Phi$ is onto. We now show that $\ker \Phi = T_2 \cap \Gamma^+(G,\pi)$.

Let $\alpha_z \in \ker \Phi$. Then $c_{\pi(z)} = \id_G$, i.e.\ $\pi(z) g \pi(z)^{-1} = g$ for all $g \in G$. Hence $\pi \alpha_z(w) = \pi(z) \pi(w) \pi(z)^{-1} = \pi(w)$ for all $w \in F_2$ so that $\pi \alpha_z = \pi$. This shows that $\alpha_z \in T_2 \cap \Gamma^+(G,\pi)$.

Now suppose that $z \in F_2$ such that $\alpha_z \in T_2 \cap \Gamma^+(G,\pi)$. Then $\pi \alpha_z = \pi$ so that $\pi(z) \pi(w) \pi(z)^{-1} = \pi(w)$ for all $w \in F_2$. Since $\pi$ is onto, it follows that $\pi(z) \in \z(G)$. Hence $c_{\pi(z)} = \id_G$, i.e.\ $\alpha_z \in \ker \Phi$.
\end{proof}

\section{Congruence Subgroups associated to Abelian Groups} \label{sec-abelian}
Let $G$ be a finite abelian group and, as before, $\pi : F_2 \rightarrow G$ an epimorphism. Observe that this implies that $G \cong \Z/m\Z \times \Z / n\Z$ where $n \mid m$. Our aim in this section is to prove Theorem~\ref{thm-abel}. From Corollary~\ref{intr-cor-index-n2} we obtain

\begin{equation}\label{same_index}
[\Aut^+(F_2) : \Gamma^+(G,\pi)] = [\SL_2(\Z) : \rho(\Gamma^+(G,\pi))].
\end{equation}

We thus only have to understand the image $\rho(\Gamma^+(G,\pi))$. It is known that the action of $\Aut(F_2)$ on $\mathbf{R}_2(G)$ is transitive for abelian groups $G$. See for example \cite{NN}. As we shall see now, already the $\Aut^+(F_2)$-action on this set is transitive. Hence we only need to understand $\rho(\Gamma^+(G,\pi))$ for a single epimorphism $\pi$.

\begin{lem}\label{LemTrans}
Let $\pi : F_2 \rightarrow G$ be an epimorphism of $F_2$ onto a finite abelian group. Then $\Aut^+(F_2)$ acts transitively on $\mathbf{R}_2(G)$.

In particular, up to conjugation, $\Gamma^+(G,\pi)$ only depends on $G$ but not on the particular epimorphism $\pi$.
\end{lem}

\begin{proof}
We only prove the lemma for $G \cong \Z /m \Z \times \Z / n\Z$ with $1 \not= n\mid m$. The proof for cyclic groups is very similar.

Observe that if $\pi(x)=g_1$ and $\pi(y) = g_2$, then
$$\pi u(x) = g_1 g_2, \quad \pi u(y) = g_2,\quad \pi v(x) = g_1, \quad \pi v(y) = g_1 g_2.$$

Let us recall the basic fact that for $a,b \in \Z$, we have $\langle [a], [b] \rangle = \langle [\gcd(a,b)] \rangle$, where $[z]$ denotes the image of an integer $z$ in $\Z / m\Z$. By a slight abuse of notation we shall omit the brackets $[\ ]$ in what follows.

We write $m = kn$. Note that $G \cong \Z / kn\Z \times k\Z / kn \Z$. Let $\pi : F_2 \rightarrow \Z / kn \Z \times k\Z / kn\Z$ be an epimorphism. Write
$$\pi(x) = \left(\begin{smallmatrix} a \\ b  \end{smallmatrix}\right) \quad \mbox{and} \quad \pi(y) = \left(\begin{smallmatrix} c \\ d  \end{smallmatrix}\right).$$
It suffices to show that $\pi$ lies in the same $\Aut^+(F_2) \times \Aut(G)$-orbit as $\pi_0$ where
$$\pi_0(x) = \left(\begin{smallmatrix} 1 \\ 0  \end{smallmatrix}\right) \quad \mbox{and} \quad \pi_0(y) = \left(\begin{smallmatrix} 0 \\ k  \end{smallmatrix}\right).$$
Observe that $\langle a, c \rangle = \Z/kn\Z$. Using $u$ and $v$ (see Subsection~\ref{subsec-repres}) we can thus apply an Euclidean algorithm to $a$ and $c$ to obtain some $\varphi \in \Aut^+(F_2)$ such that
$$\pi \varphi (x) = \left(\begin{smallmatrix} \varepsilon \\ b' \end{smallmatrix}\right) \quad \mbox{and} \quad \pi \varphi (y) = \left(\begin{smallmatrix} 0 \\ d' \end{smallmatrix}\right)$$
with $\varepsilon \in (\Z / kn \Z)^*$. Now observe that $\langle \left(\begin{smallmatrix} \varepsilon \\ b'  \end{smallmatrix}\right), \left(\begin{smallmatrix} 0 \\ d'  \end{smallmatrix}\right)\rangle = \Z / kn\Z \times k\Z / kn \Z$. In particular there are $\alpha_1, \alpha_2 \in \Z / kn \Z$ such that $\alpha_1 \left(\begin{smallmatrix} \varepsilon \\ b'  \end{smallmatrix}\right) + \alpha_2 \left(\begin{smallmatrix} 0 \\ d'  \end{smallmatrix}\right) = \left(\begin{smallmatrix} 0 \\ k  \end{smallmatrix}\right)$. For these we find $\alpha_1 \varepsilon = 0$ so that $\alpha_1 = 0$. Moreover this shows that $\alpha_2 d' = k$. Hence $\langle d' \rangle = k\Z / kn\Z$, i.e.\ $\ord(d') = n$. We can thus find a suitable power $u^e$ of $u$ such that
$$\pi \varphi u^e (x) = \left(\begin{smallmatrix} \varepsilon \\ 0 \end{smallmatrix}\right) \quad \mbox{and} \quad \pi \varphi u^e (y) = \left(\begin{smallmatrix} 0 \\ d' \end{smallmatrix}\right).$$
Since $\ord(\varepsilon) = kn = \ord(1)$ and $\ord(d') = n = \ord(k)$ we can define an automorphism $\phi$ of $\Z / kn\Z \times k\Z / kn\Z$ by $\phi ( \left(\begin{smallmatrix} \varepsilon \\ 0  \end{smallmatrix}\right) )= \left(\begin{smallmatrix} 1 \\ 0  \end{smallmatrix}\right)$ and $\phi ( \left(\begin{smallmatrix} 0 \\ d'  \end{smallmatrix}\right) )= \left(\begin{smallmatrix} 0 \\ k  \end{smallmatrix}\right)$. Then $\phi \pi \varphi u^e = \pi_0$ and the lemma follows.
\end{proof}

For cyclic groups we shall choose the epimorphism
$$\pi : F_2 \longrightarrow \Z / m\Z, \quad x \longmapsto 1,\quad y \longmapsto 0.$$
It is easily seen that then
\begin{equation}\label{image1}
\rho(\Gamma^+(\Z /m \Z, \pi)) = \Gamma^1(m).
\end{equation}
For groups of the form $\Z / m \Z \times \Z/ n \Z$ where $n \mid m$ we choose
$$\pi : F_2 \longrightarrow \Z / m \Z \times \Z/ n \Z,\quad x \longmapsto \left(\begin{smallmatrix} 1 \\ 0  \end{smallmatrix}\right), \quad y \longmapsto \left(\begin{smallmatrix} 0 \\ 1  \end{smallmatrix}\right).
$$
Then
\begin{equation}\label{image2}
\rho(\Gamma^+(\Z /m\Z \times \Z/n\Z,\pi)) = \Gamma(m,n).
\end{equation}
We can now easily obtain Theorem~\ref{thm-abel} as follows. By the above lemma, $[\Aut^+(F_2) : \Gamma^+(G,\pi)]$ is independent of the choice of $\pi$. Moreover, by (\ref{same_index}) this index is equal to $[\SL_2(\Z) : \rho(\Gamma^+(G,\pi))]$. Lemma~\ref{ind_gamma} together with (\ref{image1}) and (\ref{image2}) provides the desired formulas.

The results in this section lead to a general formula for the indices of the congruence subgroups $\Gamma(a,b)$ with arbitrary $a, b \in \N$. 
\begin{cor}Let $a,b \in \N$. Then
$$[\SL_2(\Z): \Gamma(a,b)]=  n m^2 \prod_{\substack{p \mid m\\ p \ \mathrm{prime}}} (1-\frac{1}{p^2}),$$ 
where $m=\lcm(a,b)$, $n=\gcd(a,b)$. 
\end{cor} 

\begin{proof} 
The group $\Gamma(a,b)$ occurs as the image under $\rho$ of $\Gamma^+(G,\pi)$ where $G=\Z /a\Z \times \Z /b\Z$. 
Since $G\cong \Z /m\Z \times \Z /n\Z$, we obtain $[\SL_2(\Z): \Gamma(a,b)]= [\SL_2(\Z): \Gamma(m,n)]$. Now apply Lemma~\ref{ind_gamma}.
\end{proof}

\section{Congruence Subgroups associated to Dihedral Groups}\label{sec-dihedral}
Let $n \geq 3$. A presentation of the dihedral group $D_n$ is given by
$$D_n = \langle r,s \mid r^n = 1,\ s^2 = 1,\ rs = sr^{-1} \rangle.$$
The group contains exactly $2n$ elements, namely
$$1, r, r^2,\dots,r^{n-1},s,sr,sr^2,\dots,sr^{n-1}.$$
If $n$ is odd, the center $\z(D_n)$ of $D_n$ is trivial. For even $n$ its center has order $2$ and we have $\z(D_n) = \langle r^{\frac{n}{2}} \rangle$.
We choose the epimorphism
$$ \pi_0 : F_2  \longrightarrow  D_n,\quad x \longmapsto r,\ y \longmapsto s$$
and consider $\Gamma^+(D_n,\pi_0)$. By the following result this already covers the general case.
\begin{lem} \label{Dn-pi-unabh}
The action of $\Aut^+(F_2)$ on the set $\mathbf{R}_2(D_n)$ is transitive.
\end{lem}
\begin{proof}
An arbitrary epimorphism of $F_2$ onto $D_n$ can have one of the following forms.
\begin{align*}
\pi_1 :     F_2 &\rightarrow D_n \quad  & \pi_2 : F_2   &\rightarrow D_n  \quad & \pi_3 : F_2 &\rightarrow D_n\\
                  x &\mapsto r^k                        &          x        &\mapsto sr^k                   & x   &\mapsto sr^k\\
                  y &\mapsto sr^l                   &          y        &\mapsto  r^l                   & y   &\mapsto sr^l
\end{align*}
with suitable $k,l \in \Z$. Let us first consider the type $\pi_1$. Observe that $(r^k)^n = (sr^l)^2 = 1$ and $r^k sr^l = sr^l (r^k)^{-1}$. We may thus define an endomorphism $\phi : D_n \rightarrow D_n$ by $\phi(r) := r^k$ and $\phi(s) := sr^l$. Since $\langle r^k, sr^l \rangle = D_n$, this endomorphism is onto. Hence $\phi$ is an automorphism of $D_n$. It follows that $\pi_1 = \phi\pi_0$ and hence $\ker(\pi_1) = \ker(\pi_0)$. Now we consider an epimorphism of the form $\pi_2$. Let $\varphi$ be the automorphism of $F_2$ given by $\varphi(x) := y^{-1}$ and $\varphi(y) := x$. Then $\varphi \in \Aut^+(F_2)$ and $\pi_2\varphi(x) = r^{-l}$, $\pi_2\varphi(y) = sr^k$. Now $\pi_2\varphi$ is an epimorphism of the form $\pi_1$. Hence $\ker(\pi_2\varphi) = \ker(\pi_0)$, that is $\ker(\pi_2) = \varphi(\ker(\pi_0))$. Let $u \in \Aut^+(F_2)$ as in Subsection~\ref{subsec-repres}. Observe that $\pi_3u (x) = r^{l-k}$ and $\pi_3u(y)=sr^l$ so that $\pi_3u$ is again of the form $\pi_1$. We can thus argue as before.
\end{proof}

Let us now consider the index $[\Aut^+(F_2) : \Gamma^+(D_n,\pi)]$. By Corollary~\ref{intr-cor-index-n2} we have
$$[\Aut^+(F_2) : \Gamma^+(D_n,\pi)] = [\SL_2(\Z) : \rho(\Gamma^+(D_n,\pi))] \cdot [D_n : \z(D_n)].$$ Note that
$$[D_n : \z(D_n)] = \begin{cases}2n, \mbox{ if $n$ is odd}\\ n, \mbox{ if $n$ is even.} \end{cases}$$
Next we show that the image $\rho(\Gamma^+(D_n,\pi)) $ is conjugate to $ \Gamma_1(2)$, if $n$ is odd and to $\Gamma(2),$ if $n$ is even. As before, let $\pi_0 : F_2 \rightarrow D_n$ be the epimorphism defined by $ \pi_0(x)=r$ and $\pi_0(y)=s$. Lemma~\ref{Dn-pi-unabh} yields that every $\Gamma^+(D_n,\pi)$ is conjugate to $\Gamma^+(D_n,\pi_0)$, so we only need to consider the image of $\Gamma^+(D_n,\pi_0)$ under $\rho$.

Let $u$, $v$ and $\alpha_x \in \Aut^+(F_2)$ be as defined in Subsection~\ref{subsec-repres}.  Observe that the following automorphisms are in $\Gamma^+(D_n,\pi_0)$:
\begin{align*}
u^2 &= \begin{cases} x \mapsto xy^2 \\ y \mapsto y \end{cases}  \quad&
v^n &= \begin{cases} x \mapsto x \\ y \mapsto x^ny \end{cases} \\
\alpha^{-1}_x v^2&= \begin{cases} x \mapsto x \\ y \mapsto xyx \end{cases}  \quad&
\alpha^{-1}_x (u^{-1} v)^3 &= \begin{cases} x \mapsto y^{-1}x^{-1}y \\ y \mapsto y^{-1} \end{cases}
\end{align*}
The images of the above automorphisms under $\rho$ are given by
$\left( \begin{smallmatrix} 1 & 0 \\ 2 & 1 \end{smallmatrix}\right)$,
$\left( \begin{smallmatrix} 1 & n \\ 0 & 1 \end{smallmatrix}\right)$,
$\left(\begin{smallmatrix} 1 & 2 \\ 0 & 1 \end{smallmatrix}\right)$ and
$\left(\begin{smallmatrix} -1 & 0 \\ 0 & -1 \end{smallmatrix}\right)$, respectively.
For $n$ odd we thus have
$$\rho(\Gamma^+(D_n,\pi_0)) \geq \langle
\left( \begin{smallmatrix} 1 & 0 \\ 2 & 1 \end{smallmatrix}\right),\
\left( \begin{smallmatrix} 1 & 1 \\ 0 & 1 \end{smallmatrix}\right) \rangle = \Gamma_1(2)$$
and for $n$ even we have
$$\rho(\Gamma^+(D_n,\pi_0)) \geq \langle
\left( \begin{smallmatrix} -1 & 0 \\ 0 & -1 \end{smallmatrix}\right),\
\left( \begin{smallmatrix} 1 & 0 \\ 2 & 1 \end{smallmatrix}\right),\
\left( \begin{smallmatrix} 1 & 2 \\ 0 & 1 \end{smallmatrix}\right) \rangle = \Gamma(2).$$
Moreover we know from  Subsection~\ref{reduc} that
$\rho(\Gamma^+(D_n,\pi_0)) $ is a subgroup of $\rho(\Gamma^+(D_n^{\ab},\bar{\pi}_0)) $, where  $\bar{\pi}_0$ is the epimorphism $\pi_0$ followed by the natural projection onto the abelian quotient $D_n^{\ab}=D_n/ D'_n$. We have
\begin{align*}
D_n^{\ab} &= \langle \bar s \mid 2\bar s = 0 \rangle \cong \Z/2\Z, \quad &\mbox{for $n$ odd},\\
D_n^{\ab} &= \langle \bar r, \bar s \mid 2\bar r = 0, 2\bar s = 0, \bar r+\bar s=\bar s+r \rangle \cong (\Z/2\Z)^2 \quad &\mbox{for $n$ even}
\end{align*}
where $\bar r$ and $\bar s $ are the images of $r$ and $s$ in $D_n^{\ab}$. From Section~\ref{sec-abelian} we know $\rho(\Gamma^+(\Z/2\Z,\bar{\pi}_0))= \Gamma_1(2)$ and $\rho(\Gamma^+((\Z/2\Z)^2,\bar{\pi}_0))= \Gamma(2)$ and hence
$$\rho(\Gamma^+(D_n,\pi_0)) = \begin{cases}\Gamma_{1}(2), \mbox{ if $n$ is odd}\\
\Gamma(2), \mbox{ if $n$ is even.}\end{cases}$$
By Lemma~\ref{ind_gamma} we thus have
$$[\SL_2(\Z) : \rho(\Gamma^+(D_n,\pi))] = \begin{cases} 3, \mbox{ if $n$ is odd} \\ 6, \mbox{ if $n$ is even.} \end{cases}$$
Altogether we find that $[\Aut^+(F_2) : \Gamma^+(D_n,\pi_0)] = 6n$, which proves the first part of Theorem~\ref{thm-dieder}.

In the above calculation we used four automorphisms contained in $\Gamma^+(D_n,\pi_0)$. Now we show that these actually generate $\Gamma^+(D_n,\pi_0)$, thereby proving the second part of Theorem~\ref{thm-dieder}.

\begin{prop}
The group $\Gamma^+(D_n,\pi_0)$ is generated by the four automorphisms $u^2,\;  v^n,\; \alpha^{-1}_x v^2,$ and $ \alpha^{-1}_x (u^{-1} v)^3$.
\end{prop}

\begin{proof}
The main strategy of the proof is to compute generators of $\Gamma^+(D_n,\pi_0)$ using the Reidemeister method \cite[Theorem~2.7]{MKS} and then show that each generator can be written as a product of $u^2$, $v^n$, $\alpha^{-1}_x v^2$ and $\alpha^{-1}_x (u^{-1} v)^3$.

Recall the following exact sequence
$$1 \longrightarrow  T_2 \cap \Gamma^+(D_n,\pi_0) \longrightarrow  \Gamma^+(D_n,\pi_0) \stackrel{\rho}{\longrightarrow}  \rho(\Gamma^+(D_n,\pi_0)) \longrightarrow 1.$$
By this sequence $\Gamma^+(D_n,\pi)$ is generated by the generators of $ T_2 \cap \Gamma^+(D_n,\pi_0)$ together with preimages of the generators of $\rho(\Gamma^+(D_n,\pi_0))$.

We first consider the case where $n$ is odd. In this case the center of $D_n$ is trivial. So $\Inn(D_n) \cong D_n$ and thus Lemma~\ref{LemCen} yields an isomorphism
$$T_2 \cap \Gamma^+(D_n,\pi) \backslash T_2 \stackrel{\cong}{\longrightarrow} D_n,\quad [\alpha_w] \longmapsto \pi(w).$$
Hence a set of right coset representatives of $T_2 \cap \Gamma^+(D_n,\pi)$ in $ T_2$ is given by
$$\id_{F_2},\ \alpha_x,\ \alpha_{x^2},\ \dots,\  \alpha_{x^{n-1}},\  \alpha_y,\ \alpha_{yx},\ \dots,\ \alpha_{yx^{n-1}}.$$
We can now use the Reidemeister method to find that $T_2 \cap \Gamma^+(D_n,\pi)$ is freely generated by
\begin{align*}
&\alpha_{x^n},\ \alpha_{y^2},\ \alpha_{yx^ny^{-1}},\\
&\alpha_{x^kyx^{k-n}y},\ \alpha_{yx^kyx^{k-n}}\ (1 \leq k \leq n-1).
\end{align*}
In the above computation we already showed that
$$\rho(\Gamma^+(D_n,\pi_0))= \Gamma_{1}(2)=\langle \left( \begin{smallmatrix} 1 & 0 \\ 2 & 1 \end{smallmatrix}\right),\
\left( \begin{smallmatrix} 1 & 1 \\ 0 & 1 \end{smallmatrix}\right)\rangle.$$
Let
$$\varphi_{1} = \begin{cases} x \mapsto xy^2 \\ y \mapsto y \end{cases} \quad \mbox{and} \quad
\varphi_{2} = \begin{cases} x \mapsto x \\ y \mapsto x^{\frac{1-n}{2}}yx^{\frac{n+1}{2}} \end{cases}$$
so that $\rho(\varphi_{1})=\left( \begin{smallmatrix} 1 & 0 \\ 2 & 1 \end{smallmatrix}\right)$ and
$\rho(\varphi_{2})=\left( \begin{smallmatrix} 1 & 1 \\ 0 & 1 \end{smallmatrix}\right)$.
An easy computation shows that these are elements of $\Gamma^+(D_n,\pi_0)$.
Hence $\Gamma^+(D_n,\pi_0)$ is generated by
\begin{align*}
&\varphi_{1},\ \varphi_{2},\ \alpha_{x^n},\ \alpha_{y^2},\ \alpha_{yx^ny^{-1}},\\
&\alpha_{x^kyx^{k-n}y},\ \alpha_{yx^kyx^{k-n}}\ (1 \leq k \leq n-1).
\end{align*}
To ease notation we set
$$\gamma_1:=u^2,  \quad \gamma_2:=v^n, \quad \gamma_3:=\alpha^{-1}_x v^2 \quad \mbox{and}\quad \gamma_4:=\alpha^{-1}_x (u^{-1} v)^3.$$
It is elementary to verify that
$$\begin{array}{lllllll}
\varphi_{1} & = & \gamma_1,  &  &
\alpha_{x^n} & = & \gamma^2_2 \gamma^{-n}_3, \\
\varphi_{2} & = &  \gamma^{-1}_2 \gamma_3^{\frac{n+1}{2}},  &  &
\alpha_{y^2} & = &\gamma^{-1}_1  \gamma^{-1}_4  \gamma_1  \gamma_4 ,\\
\alpha_{x^kyx^{k-n}y}& = & \gamma^k_3\gamma_4 \alpha^{-1}_{y^2} \alpha_{x^n}\gamma^{-k}_3\gamma_4,& &
\alpha_{yx^ny^{-1}}  & =& \gamma^{-1}_4\alpha^{-1}_{y^2}  \alpha^{-1}_{x^n}   \alpha_{y^2} \gamma_4 ,\\
\alpha_{yx^kyx^{k-n}}& =&\alpha_{y^2}\gamma_4 \gamma^k_3 \gamma^{-1}_4 \gamma^{-k}_3 \alpha^{-1}_{x^n}. & & & & \\
\end{array}$$

Now we consider the case where $n$ is even. In this case the center of $D_n$ is cyclic of order 2, generated by $r^{\frac{n}{2}}. $
By Lemma~\ref{LemCen} we have an isomorphism
$$T_2 \cap \Gamma^+(D_n,\pi) \backslash T_2 \stackrel{\cong}{\longrightarrow} \z(D_n)  \backslash D_n \cong D_{\frac{n}{2}}.$$
Analogous to the previous case we obtain that $T_2 \cap \Gamma^+(D_n,\pi)$ is freely generated by
\begin{align*}
&\alpha_{x^{\frac{n}{2}}},\ \alpha_{y^2},\ \alpha_{yx^{\frac{n}{2}}y^{-1}},\\
&\alpha_{x^kyx^{k-\frac{n}{2}}y},\ \alpha_{yx^kyx^{k-\frac{n}{2}}}\ (1 \leq k \leq \frac{n}{2}-1).
\end{align*}
Furthermore, we have seen above that
$$\rho(\Gamma^+(D_n,\pi_0))=
\Gamma(2)=\langle
\left( \begin{smallmatrix} 1 & 0 \\ 2 & 1 \end{smallmatrix}\right),
\left( \begin{smallmatrix} 1 & 2 \\ 0 & 1 \end{smallmatrix}\right),
\left( \begin{smallmatrix} -1 & 0 \\ 0 & -1 \end{smallmatrix}\right)
\rangle.$$
The automorphisms
$$\varphi_{1} = \begin{cases} x \mapsto xy^2 \\ y \mapsto y \end{cases}, \quad
\varphi_{3} = \begin{cases} x \mapsto x \\ y \mapsto xyx \end{cases} \quad \mbox{and} \quad
\varphi_{4} = \begin{cases} x \mapsto y^{-1}x^{-1}y \\ y \mapsto y^{-1} \end{cases}$$
are in $\Gamma^+(D_n,\pi_0)$ and also preimages of the generators of $\Gamma(2)$.
So $\Gamma^+(D_n,\pi)$ is generated by
\begin{align*}&\varphi_{1},\
\varphi_{3},\
\varphi_{4},\
\alpha_{x^{\frac{n}{2}}},\
\alpha_{y^2},\
\alpha_{yx^{\frac{n}{2}}y^{-1}},\\
&\alpha_{x^kyx^{k-\frac{n}{2}}y},\
\alpha_{yx^kyx^{k-\frac{n}{2}}}\ (1 \leq k \leq \frac{n}{2}-1).
\end{align*}
Similarly to the previous case we can write
$$\begin{array}{lllllll}
\varphi_{1} & = & \gamma_1, & \quad  & \alpha_{x^{\frac{n}{2}}} & = & \gamma_2 \gamma^{-\frac{n}{2}}_3 , \\
\varphi_{3} & = & \gamma_3, &  & \alpha_{y^2}  & = & \gamma^{-1}_1  \gamma^{-1}_4  \gamma_1  \gamma_4 , \\
\varphi_{4} & = & \gamma_4, &  & \alpha_{yx^{\frac{n}{2}}y^{-1}}  & = & \gamma^{-1}_4\alpha^{-1}_{y^2}  \alpha^{-1}_{x^{\frac{n}{2}}}   \alpha_{y^2} \gamma_4,\\
& & & & \alpha_{x^kyx^{k-\frac{n}{2}}y} & = & \gamma^k_3\gamma_4 \alpha^{-1}_{y^2} \alpha_{x^{\frac{n}{2}}}\gamma^{-k}_3\gamma_4,\\
& & & & \alpha_{yx^kyx^{k-\frac{n}{2}}}  & = & \alpha_{y^2}\gamma_4 \gamma^k_3 \gamma^{-1}_4 \gamma^{-k}_3 \alpha^{-1}_{x^{\frac{n}{2}}}. 
\end{array}$$
This completes the proof.
\end{proof}

\section{A Remark on the Congruence Subgroup Problem}\label{sec-congr}
Let $G$ be a finite group and $\pi:F_2 \rightarrow G$ be an epimorphism. As we have seen in Sections 3 and 4, the image $\rho(\Gamma^+(G,\pi))$ is a congruence subgroup of $\SL_2(\Z)$, if $G$ is abelian or dihedral. One might expect that, more general, $\rho(\Gamma^+(G,\pi))$ is a congruence subgroup if $G$ is solvable. We now show that this is false. 
\begin{prop}
There is a solvable group $G$ and an epimorphism $\pi:F_2 \rightarrow G$ such that $\rho(\Gamma^+(G,\pi))$ is not a congruence subgroup of $\SL_2(\Z)$.
\end{prop}
A connection to the congruence subgroup problem for $\Aut^+(F_2)$ is given by
\begin{cor}There is a finite-index subgroup of $\Aut^+(F_2)$ which does not contain any $\Gamma^+(G,\pi)$, where $G$ is abelian or dihedral.
\end{cor}

Let us explain how one can verify the above proposition. All computations in what follows were carried out by MAGMA. Computer experiments show that for $G$ solvable of order less than $128$, the image $\rho(\Gamma^+(G,\pi))$ is a congruence subgroup of $\SL_2(\Z)$. Note that these groups $G$ are metabelian. There are exactly four non metabelian solvable groups of order $128$ which can be generated by two elements. One of them, call it $G$, admits a presentation on the generators $g_1$, $g_2$, $g_3$, $g_4$, $g_5$, $g_6$, $g_7$ subject to the following relations
\begin{center}
$\begin{array}{llll}
g_1^2 = g_4 ,&g_2^{g_1} = g_2   g_3  ,&g_5^{g_2} = g_5   g_7,  & g_6^{g_5} = g_6, \\ 
g_2^2 = 1 ,&g_3^{g_1} = g_3   g_5  ,&g_5^{g_3} = g_5   g_7,  &g_7^{g_1} = g_7,  \\ 
g_3^2 = 1 ,&g_3^{g_2} = g_3  ,&g_5^{g_4} = g_5   g_7,  &g_7^{g_2} = g_7,  \\ 
g_4^2 = 1 ,&g_4^{g_1} = g_4  ,&g_6^{g_1} = g_6   g_7,  &g_7^{g_3} = g_7,  \\ 
g_5^2 = g_7 ,&g_4^{g_2} = g_4   g_5   g_7  ,&g_6^{g_2} = g_6   g_7,  & g_7^{g_4} = g_7, \\ 
g_6^2 = 1 ,&g_4^{g_3} = g_4   g_6   g_7  ,&g_6^{g_3} = g_6,  &g_7^{g_5} = g_7,  \\ 
g_7^2 = 1 ,&g_5^{g_1} = g_5   g_6  ,&g_6^{g_4} = g_6,  & g_7^{g_6} = g_7.
\end{array} $
\end{center}
One can verify that $G$ is generated by $g_1$ and $g_2$. The  commutator subgroup $G'$ is generated by $g_3, g_5, g_6, g_7 $. Further, $[G',G']$ is generated by $ g_7 $. Hence $G$ is solvable and has derived length $3$. 
We choose the epimorphism
$$\pi: F_2  \longrightarrow  G,\quad x \longmapsto g_1,\ y \longmapsto g_2.$$
Now we compute generators of $\Gamma^+(G,\pi)$. To this end we choose random elements $\varphi \in \Aut^+(F_2)$ and and collect those for which $\pi \varphi = \pi$ in a set $M$ until $M$ generates a finite-index subgroup of $\Aut^+(F_2)$. Let $u$, $v$, $\alpha_x$, $\alpha_y \in \Aut^+(F_2)$ be as in Subsection \ref{subsec-repres} and set $p:=\alpha_x^{-1} u^{-1}vu^{-1}$ and $q:= \alpha_x^{-2}u^{-1}vu^{-2}$. By the above process, we obtain $[\Aut^+(F_2): \langle M\rangle ] = 6144$ where $M$ is the set given in Table~\ref{Table}.
\begin{table}
\begin{footnotesize}
$\begin{array}{ll}
(p   q ^{-1})^4,\quad 
(q ^{-1}   p)^2,\\

(p ^{-1}   q ^{-1}   p^2)^4,\quad
p ^{-2}   q ^{-1}   p   q ^{-1}   p ^{-1},\\

p ^{-1}   q ^{-1}   p ^{-1}   q   p ^{-1}   q   p ^{-1}   q ^{-1}   p ^{-1}   q   p ^{-2}   q   p^2   q   p   q ^{-1}   p ^{-1}   q   p ^{-1}   q   p ^{-1}   q   p   q ^{-1},\\

p ^{-1}   q ^{-1}   p ^{-1}   q   p   q ^{-1}   p   q ^{-1}   p   q ^{-1}   p^2   q   p ^{-1}   q   p ^{-1}   q   p ^{-1}   q ^{-1}   p   q   p ^{-1},\\

q   p   q   p ^{-1}   q ^{-1}   p   q ^{-1}   p ^{-1}   q   p ^{-2}   q ^{-1}   p   q ^{-1}   p   q   p   q   p   q   p   q   p   q ^{-1}   p^2,\\

p ^{-1}   q ^{-1}   p ^{-1}   q   p ^{-1}   q   p ^{-1}   q   p ^{-2}   q ^{-1}   p^2   q   p ^{-1}   q   p ^{-1}   q   p ^{-1}   q   p   q ^{-1},\\

q   p ^{-1}   q ^{-1}   p ^{-1}   q ^{-1}   p   q ^{-1}   p ^{-1}   q   p ^{-2}   q   p ^{-1}   q   p^2   q   p ^{-1}   q   p ^{-1}   q ^{-1}   p ^{-1}   q ^{-1},\\

p ^{-1}   q ^{-1}   p ^{-1}   q   p   q ^{-1}   p   q ^{-1}   p   q   p ^{-1}   q   p   q   p   q   p   q ^{-1}   p   q ^{-1}   p   q   p   q ^{-1},\\

q ^{-1}   p ^{-1}   q   p ^{-1}   q ^{-1}   p ^{-1}   q   p ^{-1}   q ^{-1}   p   q ^{-1}   p   q   p^2   q   p   q ^{-1}   p ^{-1}   q   p   q ^{-1}   p ^{-1}   q ^{-1},\\

q   p ^{-1}   q ^{-1}   p ^{-1}   q ^{-1}   p ^{-1}   q   p   q ^{-1}   p^2   q   p ^{-1}   q ^{-1}   p   q   p ^{-1}   q   p   q ^{-1},\\

p ^{-1}   q ^{-1}   p   q   p ^{-1}   q ^{-1}   p ^{-1}   q   p ^{-2}   q   p ^{-2}   q   p^2   q ^{-1}   p^2   q ^{-1}   p^2   q ^{-1}   p   q   p   q ^{-1}   p ^{-1}   q   p ^{-1},\\

q   p ^{-1}   q ^{-1}   p ^{-1}   q ^{-1}   p   q   p   q ^{-1}   p^2   q   p ^{-1}   q ^{-1}   p   q   p ^{-1}   q   p ^{-1}   q ^{-1},\\

p ^{-2}   q   p ^{-1}   q ^{-1}   p ^{-1}   q   p   q   p ^{-1}   q   p   q   p   q   p   q   p   q ^{-1}   p ^{-1}   q   p,\\

p ^{-1}   q ^{-1}   p   q   p ^{-1}   q ^{-1}   p ^{-1}   q   p   q ^{-1}   p ^{-2}   q   p ^{-1}   q ^{-1}   p^2   q ^{-1}   p^2   q ^{-1}   p^2   q ^{-1}   p ^{-1}   q   p   q ^{-1}   p   q  p,\\

p ^{-1}   q ^{-1}   p ^{-1}   q ^{-1}   p ^{-1}   q   p ^{-1}   q   p   q ^{-1}   p   q   p   q   p ^{-1}   q ^{-1}   p   q   p ^{-1}   q ^{-1}   p ^{-1}   q   p,\\

q   p   q   p ^{-1}   q ^{-1}   p ^{-1}   q   p ^{-2}   q ^{-1}   p   q   p   q ^{-1}   p^2   q ^{-1}   p^2   q ^{-1}   p   q   p   q ^{-1}   p   q,\\

p ^{-2}   q   p   q ^{-1}   p ^{-1}   q   p ^{-1}   q   p ^{-1}   q   p^2   q ^{-1}   p   q ^{-1}   p   q ^{-1}   p   q   p ^{-1}   q ^{-1},\\

p ^{-1}   q ^{-1}   p ^{-1}   q   p   q   p ^{-1}   q ^{-1}   p ^{-1}   q   p^2   q ^{-1}   p   q   p ^{-1}   q ^{-1}   p   q ^{-1}   p ^{-1}   q   p,\\

p ^{-1}   q ^{-1}   p ^{-1}   q   p ^{-1}   q   p   q ^{-1}   p ^{-1}   q   p^2   q ^{-1}   p   q   p ^{-1}   q ^{-1}   p ^{-1}   q ^{-1}   p   q   p,\\

p ^{-1}   q ^{-1}   p ^{-1}   q   p ^{-1}   q   p   q ^{-1}   p ^{-2}   q ^{-1}   p ^{-1}   q   p   q ^{-1}   p^2   q ^{-1}   p   q ^{-1}   p   q   p   q   p,\\

q   p ^{-1}   q ^{-1}   p ^{-1}   q ^{-1}   p   q   p ^{-1}   q   p ^{-1}   q   p ^{-1}   q ^{-1}   p   q   p ^{-1}   q ^{-1},\\

q   p   q ^{-1}   p ^{-1}   q   p   q   p ^{-1}   q   p   q   p   q   p   q   p   q   p   q ^{-1}   p ^{-1}   q ^{-1},\\

q ^{-1}   p ^{-1}   q   p ^{-1}   q   p ^{-1}   q   p ^{-1}   q   p ^{-1}   q   p ^{-1}   q   p ^{-1}   q   p ^{-1}   q ^{-1},\\

p ^{-1}   q ^{-1}   p ^{-1}   q   p ^{-1}   q ^{-1}   p   q   p   q ^{-1}   p   q ^{-1}   p^2   q   p ^{-1}   q   p ^{-1}   q ^{-1}   p ^{-1}   q   p   q ^{-1}   p ^{-1}   q   p,\\

q   p ^{-1}   q ^{-1}   p ^{-1}   q   p ^{-1}   q   p ^{-1}   q ^{-1}   p ^{-1}   q   p   q   p   q ^{-1}   p ^{-1}   q   p ^{-1}   q   p ^{-1}   q ^{-1}   p   q   p,\\

p ^{-1}   q ^{-1}   p ^{-1}   q   p ^{-1}   q ^{-1}   p ^{-1}   q   p ^{-1}   q   p ^{-1}   q   p ^{-1}   q ^{-1}   p   q ^{-1}   p   q   p,\\

p ^{-1}   q ^{-1}   p   q   p   q ^{-1}   p ^{-1}   q   p ^{-1}   q   p ^{-1}   q   p   q ^{-1}   p ^{-1}   q ^{-1}   p   q   p,\\

q   p   q   p ^{-1}   q ^{-1}   p   q   p ^{-1}   q ^{-1}   p ^{-1}   q ^{-1}   p   q   p^2   q ^{-1}   p   q   p ^{-1}   q   p   q ^{-1}   p   q ^{-1}   p ^{-1}   q   p
\end{array}$\vspace{1ex}\\
\end{footnotesize}
\caption{The Elements of $M$.\label{Table}}
\end{table}
Since, by construction, $\langle M \rangle \leq \Gamma^+(G,\pi)$, we have $[\Aut^+(F_2): \Gamma^+(G,\pi) ] \leq 6144$. We can compute the length of the  orbit of $\pi$ under the $\Aut^+(F_2)$-action on the set of epimorphisms $\E_2(G)$ (see Section \ref{secAct}) to obtain 
$$[\Aut^+(F_2):\Gamma^+(G,\pi) ] = |\Aut^+(F_2) \cdot \pi| = 6144$$
This shows that $\langle M \rangle = \Gamma^+(G,\pi)$. It is now easily verifed that $\rho(\Gamma^+(G,\pi))$ is generated by the elements given in Table~\ref{Table2}. Here $e_1$ and $e_2$ are the generators of $\SL_2(\Z)$ given in Subsection~\ref{subsec-repres}.

\begin{table}
\begin{footnotesize}
$\begin{array}{lll}
&e_2   e_1^2   e_2^3   e_1   e_2^2   e_1^{-1}, &e_2^4,\\
&e_2^{-2}   e_1^{-1}   e_2^{-6}   e_1   e_2^{-1}   e_1   e_2^{-1}   e_1   e_2^{-1}   e_1, 	&e_2^{-4},\\
&e_2^2   e_1   e_2^6   e_1   e_2^{-1}   e_1   e_2^{-1}   e_1   e_2^{-1}   e_1^{-1}, 	&(e_2   e_1   e_2)^4,\\
&e_2^{-1}   e_1^{-5}   e_2^{-1}   e_1^{-1}   e_2^{-1}   e_1, &e_2^{-2}   e_1^{-1}   e_2^{-4}   e_1^{-1}   e_2^{-2}   e_1^2,\\
&e_2   e_1   e_2   e_1   e_2^2   e_1   e_2^2   e_1   e_2   e_1   e_2^{-1}   e_1   e_2^{-1}   e_1   e_2^{-1}   e_1^{-1}, &e_1^2,\\
&e_2^{-1}   e_1^{-1}   e_2^{-2}   e_1^{-4}   e_2^{-2}   e_1^{-1}   e_2^{-1}   e_1   e_2^{-1}   e_1   e_2^{-1}   e_1   e_2^{-1},	&e_2^{-1}   e_1^{-6}   e_2^{-1}   e_1   e_2^{-1}   e_1   e_2^{-1}   e_1   e_2^{-1},\\
&e_2   e_1   e_2^6   e_1   e_2   e_1   e_2^{-1}   e_1   e_2^{-1}   e_1   e_2^{-1}, 	&e_2^{-2}   e_1^{-7}   e_2^{-2}   e_1   e_2^{-1}   e_1   e_2^{-1}   e_1   e_2^{-1}   e_1,\\
&e_2^{-1}   e_1^{-1}   e_2^{-1}   e_1^{-3}   e_2^{-3}   e_1   e_2^{-1}   e_1   e_2^{-1}   e_1   e_2^{-1}   e_1,
&e_1^{-1}   e_2^{-1}   e_1^{-2}   e_2^{-1}   e_1   e_2^{-1}   e_1   e_2^{-1}   e_1   e_2^{-1}   e_1,\\
&e_2^{-1}   e_1^{-6}   e_2^{-1}   e_1   e_2^{-1}   e_1   e_2^{-1}   e_1   e_2^{-1}   e_1^2
\end{array}$\vspace{1ex}\\
\end{footnotesize}
\caption{Generators of $\rho(\Gamma^+(G,\pi))$. \label{Table2}}
\end{table}

Let us assume that $\rho(\Gamma^+(G,\pi))$ is a congruence subgroup. Then we can determine the level of $\rho(\Gamma^+(G,\pi))$, which is by \cite[Lemma~2.3]{GS} the smallest positive integer $a$ such that $\rho(\Gamma^+(G,\pi))$ contains the normal closure $\langle e^a_2\rangle^{\SL_2(\Z)}$. Clearly we have $\langle e^a_2\rangle^{\SL_2(\Z)} \leq \rho(\Gamma^+(G,\pi))$ if and only if $s e^a_2 s^{-1} \in \rho(\Gamma^+(G,\pi))$, where $s$ runs through a set of coset representatives of $\rho(\Gamma^+(G,\pi))$ in $\SL_2(\Z)$. By an easy MAGMA computation we obtain the level $a=8$. Now \cite[Theorem~2.5]{GS} implies that $\rho(\Gamma^+(G,\pi))$ contains the principal congruence subgroup $\Gamma(8)$. However, 
$e_1^{-1}   e_2^{-1}   e_1^{-2}   e_2^{-2}   e_1^{-11}   e_2^{-3}   e_1^{-1}   e_2   e_1^{-1}   e_2   e_1^{-1}   e_2   e_1^{-1}   e_2=
\left(\begin{smallmatrix} -327 & -80 \\ 560 & 137\end{smallmatrix}\right)$ is obviously an element of $\Gamma(8)$ but not contained in $\rho(\Gamma^+(G,\pi))$. Hence $\Gamma(8) \not\leq \rho(\Gamma^+(G,\pi))$, contradiction. It follows that $\rho(\Gamma^+(G,\pi))$ cannot be a congruence subgroup of $\SL_2(\Z)$.

\bibliographystyle{plain}

\end{document}